\documentclass[12pt, reqno]{amsart}
\usepackage{amsmath, amsthm, amscd, amsfonts, amssymb, graphicx, color, float}
\usepackage[bookmarksnumbered, colorlinks, plainpages]{hyperref}
\input{mathrsfs.sty}
\hypersetup{colorlinks=true,linkcolor=red, anchorcolor=green, citecolor=cyan, urlcolor=red, filecolor=magenta, pdftoolbar=true}

\newtheorem{theorem}{Theorem}[section]
\newtheorem{lemma}[theorem]{Lemma}
\newtheorem{proposition}[theorem]{Proposition}
\newtheorem{corollary}[theorem]{Corollary}
\theoremstyle{definition}

\newtheorem{example}[theorem]{Example}

\theoremstyle{remark}

\newtheorem{conjecture}[theorem]{Conjecture}
\newtheorem{problem}[theorem]{Problem}
\numberwithin{equation}{section}
\begin{document}
	
	\title{A separation theorem for Hilbert $C^*$-modules}
	
	\author[R. Eskandari, M.S. Moslehian]{Rasoul Eskandari$^1$ \MakeLowercase{and} Mohammad Sal Moslehian$^2$}
	
	\address{$^1$Department of Mathematics Education, Farhangian University, P.O. Box 14665-889, Tehran, Iran.}
	\email{Rasul.eskandari@cfu.ac.ir, eskandarirasoul@yahoo.com}

	\address{$^2$Department of Pure Mathematics, Center of Excellence in Analysis on Algebraic Structures (CEAAS), Ferdowsi University of Mashhad, P.O. Box 1159, Mashhad 91775, Iran}
	\email{moslehian@um.ac.ir; moslehian@yahoo.com}

	
\begin{abstract}
Let $\mathscr E$ be a Hilbert $\mathscr A$-module over a $C^*$-algebra $\mathscr A$. For each positive linear functional $\omega$ on $\mathscr A$, we consider the localization $\mathscr E_\omega$ of $\mathscr E$, which is the completion of the quotient space $\mathscr E/\mathscr {N}_\omega$, where $\mathscr N_\omega=\{x\in \mathscr E:\omega\langle x,x\rangle=0\}$. Let $\mathscr H$ and $\mathscr K$ be closed submodules of $\mathscr E$ such that $\mathscr H\cap \mathscr K$ is orthogonally complemented, and let $\omega=\sum_{j=1}^{\infty}\lambda_j\omega_j$, where $\lambda_j>0$, $\sum_{j=1}^{\infty}\lambda_j=1$, and $\omega_j$'s are positive linear functionals on $\mathscr A$. We prove that if $(\mathscr H\cap \mathscr K)_{\omega_j}=\mathscr H_{\omega_j}\cap \mathscr K_{\omega_j}$ for each $j$, then 
\[
(\mathscr H\cap \mathscr K)_\omega=\mathscr H_\omega\cap \mathscr K_\omega\,.
\]
Furthermore, let $\mathscr L$ be a closed submodule of a Hilbert $\mathscr A$-module $\mathscr E$ over a $W^*$-algebra $\mathscr A$. We pose the following separation problem: ``Does there exist a normal state $\omega$ such that $\iota_\omega (\mathscr L)$ is not dense in $\mathscr E_\omega $?'' In this paper, among other results, we give an affirmative answer to this problem, when $\mathscr E$ is a self-dual Hilbert $C^*$-module over a $W^*$-algebra $\mathscr A$ such that $\mathscr E\backslash \mathscr L$ has a nonempty interior with respect to the weak$^*$-topology. This is a step toward answering the above problem.
\end{abstract}
	\maketitle
	\section{Introduction}
		Let $\mathscr A$ be a $C^*$-algebra and let $\mathscr E$ be a Hilbert $\mathscr A$-module. By $\mathrm{S}(\mathscr A)$ we denote the sets of all states on $\mathscr A$.		
		Let $\omega$ be a positive linear functional on $\mathscr A$. Then,
		\begin{equation}\label{inner product}
		\mathscr N_\omega=\{x\in \mathscr E: \omega\langle x,x\rangle=0\}
		\end{equation}
		is a closed subspace of $\mathscr E$. One can put an inner product $(\cdot,\cdot)_\omega$ on the quotient space $\mathscr E/\mathscr N_\omega$ by
		\[
		\left( x+\mathscr N_\omega, y+\mathscr N_\omega\right)_\omega:=\omega\langle x,y\rangle\qquad(x,y\in \mathscr E)\,.
		\]
Let $\mathscr E_\omega$ stand for the Hilbert space completion of $\mathscr E/\mathscr N_\omega$. The natural quotient map is the linear function $\iota_\omega:\mathscr E\to \mathscr E_\omega$, defined by $\iota_\omega(x)=x+\mathscr N_\omega\,\, (x\in \mathscr E)$.
Let $\pi_\omega:\mathscr A\to \mathbb{B}(\mathscr H_{\pi_\omega})$ be the G-N-S representation of $\mathscr A$ on the Hilbert space $\mathscr H_{\pi_\omega}$ associated with the state $\omega$.
Let $\mathscr E\otimes_\mathscr A\mathscr H_{\pi_\omega}$ be the Hilbert space completion of the algebraic tensor product $\mathscr E\otimes_\mathscr A\mathscr H_{\pi_\omega}$ with respect to the inner product 
 \[
 \langle x\otimes h,x'\otimes h'\rangle=\langle h,\pi\left(\langle x,x'\rangle\right) h'\rangle\quad(x,x'\in \mathscr E,h,h'\in \mathscr H_{\pi_\omega}).
 \]
Then $\mathscr E\otimes_\mathscr A\mathscr H_{\pi_\omega}$ and $\mathscr E_\omega$ are isomorphic, as shown in \cite[section 2.2]{Kaad}.\\
Let $S$ be a subset of a topological vector space. The $\sigma$-convex hull of $S$ is denoted by $\sigma\mbox{-co}(S)$ and defined as follows:
\[
\sigma\mbox{-co}(S):=\left\{\sum_{j=1}^{\infty}\lambda_js_j: s_j\in S,\lambda_j>0,\sum_{j=1}^{\infty}\lambda_j=1\right\}\,.
\]
 Consider the weak$^*$-topology on $\mathrm{S}(\mathscr A)$. It is evident that if $\mathscr A$ is a unital $C^*$-algebra and $\{\omega_j\}_{j=1}^\infty\subseteq \mathrm{S}(\mathscr A)$, then $\sigma\mbox{-co}(\{\omega_j\})\subseteq \mathrm{S}(\mathscr A)$. 

In the setting of $C^*$-algebras, the following theorems are well-known. 
\begin{theorem}\label{theorem jfa}\cite[Theorem 3.1]{Kaad}
	Let $\mathscr L$ be a closed convex subset of a Hilbert $C^*$-module $\mathscr E$ over $\mathscr A$. For
	each vector $x_0\in\mathscr E\backslash \mathscr L$ there exists a state $\omega$ on $\mathscr A$ and $y\in \mathscr E$ such that $\iota_\omega (y)$ is not in the closure of $\iota_\omega (\mathscr L)$.
	In particular, there exists a state $\omega$ such that $\iota_\omega (\mathscr L)$ is not dense in $\mathscr E_\omega $; and hence, $\iota_\omega (\mathscr L)^\perp\neq \{0\}$ when $\mathscr L$ is a closed submodule.
\end{theorem}
\begin{theorem}( \cite[Proposition 1.6]{Mes})\label{lemma Mesland}
	Let $\mathscr L \subseteq \mathscr E$ be a closed submodule.
	Then, $\mathscr L$ is complemented if and only if for every $\omega\in S(\mathscr A)$ there is an equality
	$(\mathscr L_\omega)^\perp =(\mathscr L^\perp)_\omega$.
\end{theorem}
In their study of regular operators in Hilbert $C^*$-modules, Kaad and Lesch provided a positive answer to the following conjecture:
\begin{conjecture}\label{conjecture}{\cite[Conjecture 5.2]{Kaad}}
In the situation of Theorem \ref{theorem jfa}, if $\mathscr L$ is a closed $\mathscr A$-submodule, then there exist a pure state $\omega$ and an element $x_0\in \mathscr E$ such that $\iota_\omega(x_0)$ is not in the closure of $\iota_\omega(\mathscr L)$. In particular, there exists a pure state $\omega$ such that $\iota_\omega(\mathscr L)$ is not dense in $\mathscr E_\omega$, and hence, $\iota_\omega(\mathscr L)^\perp\neq 0$. 
\end{conjecture}

In \cite{Kaad2}, a proof of Conjecture \ref{conjecture} is presented, which is based on the arguments in the
proof of \cite[Proposition 1.16]{Pri}. 
We pose the following problem concerning Conjecture \ref{conjecture}:
\begin{problem}\label{conjecture2}
	Let $\mathscr E$ be a Hilbert $\mathscr A$-module over a $W^*$-algebra $\mathscr A$. 
	Does there exist a normal state $\omega$ such that $\iota_\omega (\mathscr L)$ is not dense in $\mathscr E_\omega $?
\end{problem}

	To gain more understanding of $C^*$-algebras and Hilbert $C^*$-modules, we recommend interested readers to consult \cite{R, Sa} for the former and \cite{MT} for the latter. 

	In this note, we provide various results, including an affirmative response to Problem \ref{conjecture2}, especially Corollary \ref{esk}, under certain mild conditions. 

%=============================================================
\section{A Localization Of Hilbert $\mathscr A$-Modules On $C^*$-Algebras}
We begin this section with the following observation.

Let $\omega=\sum_{j=1}^k\lambda_j\omega_j$ with $\sum_{j=1}^k\lambda_j=1$, $\lambda_j> 0$, where $w_j$'s are positive linear functionals on a $C^*$-algebra $\mathscr{A}$.			
			 It can be verified that 
		\begin{equation}\label{eqcap}
		\mathscr N_\omega=\bigcap_{j=1}^n\mathscr N_{\omega_j}\,.
		\end{equation}
		Let us define the map
		\begin{equation}\label{remark1}
\phi_j:\mathscr E/\mathscr N_{\omega}\to \mathscr E/\mathscr N_{\omega_j},\quad \phi_j(x+\mathscr N_{\omega})=x+\mathscr N_{\omega_j}.
		\end{equation} It follows from \eqref{eqcap} that $\phi_j$ is well-defined. We have 
		\begin{equation}\label{eq norm}
		\|\phi_j(x+\mathscr N_\omega)\|^2=\|x+\mathscr N_{\omega_j}\|^2=\omega_j\langle x,x\rangle\leq \frac{1}{\lambda_j}\omega\langle x,x\rangle=\frac{1}{\lambda_j}\|x+\mathscr N_\omega\|^2.
		\end{equation}
		Hence, we can extend $\phi_j$ onto $\mathscr E_\omega$, which is represented by the same $\phi_j$ as a map from $\mathscr E_\omega$ to $\mathscr E_{\omega_j}$.
		 Note that for every $x,y\in \mathscr E$, 
		\begin{align*}
		(\phi_j(x+\mathscr N_\omega),\phi_j(y+\mathscr N_{\omega}))_{\omega_j}=(x+\mathscr N_{\omega_j},y+\mathscr N_{\omega_j})_{\omega_j}
		=\omega_j\langle x,y\rangle\quad(1\leq j\leq k)\,,
		\end{align*}
which yields that
 		\begin{align*}
		\sum_{j=1}^k\lambda_j(\phi_j(x+\mathscr N_\omega),\phi_j(y+\mathscr N_{\omega}))_{\omega_j}=\sum_{j=1}^k\lambda_j\omega_j\langle x,y\rangle=\omega\langle x,y\rangle =(x+\mathscr N_\omega,y+\mathscr N_\omega)_\omega\,.
		\end{align*}
		Since $\mathscr E/\mathscr N_\omega$ is dense in $\mathscr E_\omega$, we arrive at 
		\begin{equation}\label{convex state}
		(\tilde{u},\tilde{u}')_\omega=\sum_{j=1}^k\lambda_j(\phi_j(\tilde{u}),\phi_j(\tilde{u}'))_{\omega_j}\qquad(\tilde{u},\tilde{u}'\in \mathscr E_\omega)\,.
		\end{equation}

		The following example shows that $\overline{\iota_\omega(\mathscr L)}$ differs from 
		\[
		\{\tilde{z}\in \mathscr E_\omega: \phi_j(\tilde{z})\in \overline{\iota_{\omega_j}(\mathscr L)},1\leq j\leq n\}.
		\]
						
\begin{example}
		Consider the set $\{1,2\}$ with the discrete metric and consider the unital $C^*$-algebra $\mathscr A := \mathrm {C}( \{1,2\} )$ as a Hilbert $C^*$-module $\mathscr E$ over itself. Take into account the indicator functions $p_1$ and $p_2$ associated with the points 1 and 2 in $\{1,2\}$, i.e.,
		\[
p_1(x)=\begin{cases}
1&\quad x=1\\
0&\quad x=2
\end{cases}\qquad p_2(x)=\begin{cases}
0&\quad x=1\\
1&\quad x=2
\end{cases}
\]
		 and define the closed convex subset
		$
		\mathscr L := \{ \lambda ( p_1 + p_2) : \lambda \in \mathbb{C} \} \subseteq \mathscr E.
		$
		Consider the states $\omega_1$ and $\omega_2 : \mathscr A \to \mathbb{C}$ evaluating at the points $1$ and $2$, respectively, and put $\omega = \frac{1}{2} \omega_1 +\frac{1}{2} \omega_2$.		 	
	Since 
	\[
	\mathscr N_{\omega_1}=\{f\in \mathscr E: f(1)=0\},\quad\mathscr N_{\omega_2}=\{f\in \mathscr E: f(2)=0\},\quad \mbox{and}\quad \mathscr N_\omega=\{0\},
	\]
	from \eqref{inner product} we conclude that	the Hilbert space $\mathscr E_\omega$ coincides with $\ell^2(\{1,2\})$ under the normalized counting measure, whereas the Hilbert spaces $\mathscr E_{\omega_1}$ and $\mathscr E_{\omega_2}$ correspond to $\ell^2(\{1\})$ and $\ell^2(\{2\})$, respectively. 
		
		Moreover, it is held that $\iota_{\omega_1}(\mathscr L) =\mathscr E_{\omega_1}$ and $\iota_{\omega_2}(\mathscr L) = \mathscr E_{\omega_2}$. However,
		\begin{align*}
	\|p_1-p_2+\mathscr N_\omega\|^2=\omega\langle p_1-p_2+\mathscr N_\omega, p_1-p_2+\mathscr N_\omega\rangle =\omega (p_1+p_2+\mathscr N_\omega)=1.
		\end{align*}
		Thus, $p_1 – p_2+\mathscr N_\omega$ is a nontrivial vector in the orthogonal complement of $\iota_\omega(\mathscr L)$. Since, $\phi_1(p_1-p_2+\mathscr N_\omega)=p_1+\mathscr N_{\omega_1}\in \overline{\iota_{\omega_1}(\mathscr L)}$ and $\phi_2(p_1-p_2+\mathscr N_\omega)=-p_2+\mathscr N_{\omega_2}\in \overline{\iota_{\omega_2}(\mathscr L)}$, we arrive at 
		\[
	\overline{\iota_\omega(\mathscr L)}\subsetneq \{\tilde{z}\in \mathscr E_\omega: \phi_j(\tilde{z})\in \overline{\iota_{\omega_j}(\mathscr L)},j=1,2\}\,.
		\]
	\end{example}
	In the following theorem, we identify $\overline{\iota_\omega(\mathscr L)}$.	
		\begin{theorem}\label{thconvex}
	 Let $\mathscr E$ be a Hilbert $\mathscr A$-module. Let $\mathscr L$ be a subset of $\mathscr E$. Let $\omega,\omega_1,\ldots,\omega_n$ be positive functionals such that $\omega=\sum_{j=1}^n\lambda_j\omega_j$ with $\sum_{j=1}^n\lambda_j=1$ and $\lambda_j>0$.	
	 	Then, 
	 	\[
	 \overline{\iota_\omega(\mathscr L)}=\{\tilde{z}\in \mathscr E_\omega: {\rm there~exists~} \{x_k\}\subseteq \mathscr L {\rm ~such~that~} \lim_k(x_k+\mathscr N_{\omega_j})=\phi_j(\tilde{z}),1\leq j\leq n\}\,.
	 	\]
		\end{theorem}
	\begin{proof}
	
		Let $\tilde{z}\in \overline{\iota_\omega(\mathscr L)}$. There is a sequence $\{x_k\}\subseteq \mathscr L$ such that $$\lim_{k\to\infty}(x_k+\mathscr N_\omega)= \tilde{z}.$$
		Since $\phi_j$ is defined in \eqref{remark1} as a bounded operator, we have $\lim_{k\to\infty}\phi_j(x_k+\mathscr N_\omega)= \phi_j(\tilde{z})$. Therefore, $\lim_{k\to\infty}x_k+\mathscr N_{\omega_j}=\phi_j(\tilde{z})$ for all $1\leq j\leq n$.\\ 
		Next, 		
		let $\tilde{z}\in \mathscr E_\omega$ be such that there is a sequences $\{x_{k}\}\subseteq \mathscr L$ with 
		\begin{align}\label{eq001}
		\lim_k(x_{k}+\mathscr N_{\omega_j})=\phi_j(\tilde{z})\,\quad(1\leq j\leq n)\,.
		\end{align}
		 Then, 
		\begin{align*}
		\|x_k+\mathscr N_\omega-x_{k'}+\mathscr N_\omega\|^2&=\left\|(x_{k}-x_{k'})+\mathscr N_\omega\right\|^2\\
		&=\omega\left\langle (x_{k}-x_{k'}),(x_{k}-x_{k'})\right\rangle\\
		&=\sum_{j=1}^n\lambda_j\omega_j\langle x_{k}-x_{k'},x_{k}-x_{k'}\rangle\\
		&=\sum_{j=1}^n\lambda_j\|x_k+\mathscr N_{\omega_j}-x_{k'}+\mathscr N_{\omega_j}\|^2\,.
		\end{align*}
		This, together with \eqref{eq001}, shows that $\{x_k+\mathscr N_\omega\}$ is a Cauchy sequence. On the other hand, let $x\in \mathscr E$ be arbitrary. Then, 
		\begin{align*}
		\lim_{k\to\infty}(x_k+\mathscr N_\omega,x+\mathscr N_\omega)_\omega=\lim_{k\to\infty}\omega\langle x_k,x\rangle&=\lim_{k\to\infty}\sum_{j=1}^n\lambda_j\omega_j(\langle x_{k},x\rangle)\\
		&=\lim_{k\to\infty}\sum_{j=1}^n\lambda_j(x_{k}+\mathscr N_{\omega_j},x+\mathscr N_{\omega_j})_{\omega_j}\\
		&= \sum_{j=1}^n\lambda_j(\phi_j(\tilde{z}),\phi_j(x+\mathscr N_{\omega}))_{\omega_j}\\
		&=(\tilde{z}, x+\mathscr N_\omega)_\omega\quad(\mbox{by~\eqref{convex state}})\,.
		\end{align*}
		Hence, $x_k+\mathscr
		 N_\omega$ weakly converges to $\tilde{z}$ and since it is a Cauchy sequence we conclude that $x_k+\mathscr
		 N_\omega$ converges to $\tilde{z}$ in the norm topology. 
		\end{proof}
		\begin{lemma}\label{lemma s}
 Let $\mathscr E$ be a Hilbert $\mathscr A$-module and let $\mathscr L$ be a closed submodule of $\mathscr E$. Then, $\mathscr L_\omega$ is unitarily isomorphic to $\overline{\iota_\omega(\mathscr L)}$ for each $\omega\in \mathrm{S}(\mathscr A)$.
		\end{lemma}
		\begin{proof}
We consider the following closed subspace: 
$$\mathscr N:=\{x\in \mathscr L: \omega\langle x,x\rangle=0\}.$$
In addition, let $\phi:\iota_\omega(\mathscr L)\to \mathscr L/\mathscr N$ be such that $\phi(x+\mathscr N_\omega)=x+\mathscr N$. Then, $\phi$ is a well-defined isomorphism. In fact, 
\[
\|\phi(x+\mathscr N_\omega)\|^2=\|x+\mathscr N\|^2=\omega(\langle x,x\rangle)=\|x+\mathscr N_\omega\|^2\quad(x\in \mathscr L).
\]
The extension of $\phi$ on $\overline{\iota_\omega(\mathscr L)}$ is a unitary operator from $\overline{\iota_\omega(\mathscr L)}$ to $\mathscr L_\omega$.
		\end{proof}
	\begin{corollary}\label{esk}
Let $\mathscr H$ and $\mathscr K$ be closed submodules of $\mathscr E$ such that $\mathscr H\cap \mathscr K$ is an orthogonally complemented submodule. Let $\omega_1,\ldots,\omega_n$ be positive linear functionals on $\mathscr A$. If $(\mathscr H\cap \mathscr K)_{\omega_j}=\mathscr H_{\omega_j}\cap \mathscr K_{\omega_j}$ for each $j$, then 
\begin{equation}
(\mathscr H\cap \mathscr K)_\omega=\mathscr H_\omega\cap \mathscr K_\omega
\end{equation}
for all $\omega$ in the convex hull of $\{\omega_1,\ldots ,\omega_2\}$.
	\end{corollary}
\begin{proof}
Let $\omega$ be in the convex hull of $\{\omega_1,\ldots ,\omega_2\}$, denoted by  $\mbox{co}\{\omega_1,\ldots ,\omega_2\}$. Then, $\omega=\sum_{j=1}^n\lambda_j\omega_j$ where $\sum_{j=1}^n\lambda_j=1$. 
It is clear that $(\mathscr H\cap \mathscr K)_\omega\subseteq \mathscr H_{\omega}\cap\mathscr K_\omega$. Now, let $\tilde{z}\in \left((\mathscr H\cap\mathscr K)_\omega\right)^\perp$. By Lemma \ref{lemma Mesland}, $\tilde{z}\in \left((\mathscr H\cap\mathscr K)^\perp\right)_\omega $. It follows from Lemma \ref{lemma s} and Theorem \ref{thconvex} that there is a sequence $\{x_n\}\subseteq (\mathscr H\cap \mathscr K)^\perp$ such that for each $1\leq j\leq n$ we have
\begin{align}\label{eqlim}
\lim_{n\to \infty}(x_n+\mathscr N_{\omega_j})=\phi_j(\tilde{z}).
\end{align}
Therefore, $\phi_j(\tilde{z})\in \left((\mathscr H\cap\mathscr K)^\perp\right)_{\omega_j}$ for each $1\leq j\leq n$. We claim that $\tilde{z}\in \left(\mathscr H_\omega\cap\mathscr K_\omega\right)^\perp$. Indeed, let $\tilde{u}\in \mathscr H_\omega\cap\mathscr K_\omega$ be arbitrary. By Theorem \ref{thconvex}, there are sequences $\{h_n\}\subseteq \mathscr H$ and $\{k_n\}\subseteq\mathscr K$ such that we have 
\begin{align*}
\lim_{n\to\infty}(h_n+\mathscr N_{\omega_j})=\phi_j(\tilde{u})=\lim_{n\to\infty}(k_n+\mathscr N_{\omega_j})\qquad (1\leq j\leq n).
\end{align*}
This ensures that $\phi_j(\tilde{u})\in \mathscr H_{\omega_j}\cap\mathscr K_{\omega_j}$. By the hypothesis, $\phi_j(\tilde{u})\in (\mathscr H\cap \mathscr K)_{\omega_j}$ for each $1\leq j\leq n$. Hence, 
\[
(\phi_j(\tilde{z}),\phi_j(\tilde{u}))_{\omega_j}=0\qquad(1\leq j\leq n).
\]
Applying \eqref{convex state} we get 
\[
(\tilde{z},\tilde{u})_\omega=\sum_{j=1}^n\lambda_j(\phi_j(\tilde{z}),\phi_j(\tilde{u}))_{\omega_j}=0.
\]
Thus, $\mathscr H_\omega\cap\mathscr K_\omega\subseteq(\mathscr H\cap\mathscr K)_\omega$.
\end{proof}
In what follows, we denote the set of all bounded sequences in $\mathscr E_\omega$ by $\mathscr E_\omega^b$. 
	\begin{theorem}\label{thmain}
		Let $\mathscr E$ be a Hilbert $\mathscr A$-module, and let $\mathscr L$ be a convex subset of $\mathscr E$. Let $\omega$ and $\omega_j$'s be positive functionals such that $
		\omega=\sum_{j=1}^\infty \lambda_j\omega_j$ with $\sum_{j=1}^\infty \lambda_j=1$ and $\lambda_j>0$. 
		Then, 
		\begin{equation}\label{eqth2.4}
		\overline{\iota_\omega(\mathscr L)}=\{\tilde{z}\in \mathscr E_\omega:\exists \{x_n\}\subseteq\mathscr L,\{x_n+\mathscr N_\omega\}\in \mathscr E_\omega^b,\lim_{n\to \infty}(x_n+\mathscr N_{\omega_j})=\phi_j(\tilde{z}),~\forall j\geq 1\},
		\end{equation}
		where $\phi_j:\mathscr E_\omega\to\mathscr E_{\omega_j}$ is defined by $\phi_j(x+\mathscr N_\omega)=x+\mathscr N_{\omega_j}$ for all $x\in \mathscr E$ and for all $j\geq 1$.
	\end{theorem}
	\begin{proof} Let us denote the right-hand set in \eqref{eqth2.4} by $S$. Firstly, we demonstrate that $\overline{\iota_\omega(\mathscr L)}\subseteq S$. To do this, suppose that $\tilde{z}\in \overline{\iota_\omega(\mathscr L)}$. This implies the existence of a sequence $\{x_n\}$ in $\mathscr L$ such that $\{x_n+\mathscr N_\omega\}\in \mathscr E_\omega^b$ and $\lim_n(x_n+\mathscr N_{\omega})=\tilde{z}$. Since $\phi_j:\mathscr E_\omega\rightarrow \mathscr E_{\omega_j}$ is continuous for each $j\geq 1$, we infer that
		 $\lim_n(x_n+\mathscr N_{\omega_j})=\phi_j(\tilde{z}),j\geq 1$.
		 This shows that $\overline{\iota_\omega(\mathscr L)}\subseteq S$. 
		
		Next, let $\tilde{z}\in S$. There exists a sequence $\{x_n\}\subseteq \mathscr L$ such that $\{x_n+\mathscr N_\omega\}\in \mathscr E_\omega^b$ and 
		\begin{equation}\label{eq part part}
\lim_n(x_n+\mathscr N_{\omega_j})=\phi_j(\tilde{z})\qquad j(\geq 1).
		\end{equation}
		 The boundedness of $\{x_n+\mathscr N_\omega\}$ implies that there exists a subsequence $\{ x_{n_k}+\mathscr N_\omega\}$ that weakly converges to a vector $\tilde{u}\in \overline{\iota_\omega(\mathscr L)}$. 
		We define $\Psi:\mathscr E_\omega\longrightarrow \oplus_1^\infty \mathscr E_{\omega_j}$ by $\Psi(\tilde{x})=(\sqrt{\lambda_j}\phi_j(\tilde{x}))_j$ for $\tilde{x}\in \mathscr E_\omega$. From \eqref{eq norm}, we deduced that $\Psi$ is well-defined. The map $\Psi$ is an isometric isomorphism onto its range. In fact, for all $x,y\in \mathscr E$, we have 
		\begin{align*}
\left(x+\mathscr N_\omega,y+\mathscr N_\omega\right)_\omega=\omega(\langle x,y\rangle)&=\sum_{j=1}^{\infty}\lambda_j\omega_j\langle x,y\rangle\\
&=\sum_{j=1}^{\infty}\lambda_j\left(x+\mathscr N_{\omega_j},y+\mathscr N_{\omega_j}\right)_{\omega_j}\\
&=\sum_{j=1}^{\infty}\lambda_j\left(\phi_j(x+\mathscr N_{\omega_j}),\phi_j(y+\mathscr N_{\omega_j})\right)_{\omega_j}\\
&=\left\langle\Psi(x+\mathscr N_{\omega_j}),\Psi(y+\mathscr N_{\omega_j})\right\rangle.
		\end{align*} 
		 Hence, $\Psi(x_{n_k}+\mathscr N_\omega)$ weakly converges to $\Psi(\tilde{u})$. We claim that $\lim_{k\to\infty}(x_{n_k}+\mathscr N_{\omega_j})=\phi_j(\tilde{u})$. To establish this, let $x\in \mathscr E$ be arbitrary. Since $\Psi$ is isometric, we have 
		 \begin{align}\label{eq weak}
		 \lim_{k\to\infty}	(x_{n_k}+\mathscr N_{\omega_j},x+\mathscr N_{\omega_j})_{\omega_j}&=\lim_{k\to\infty}\left\langle \Psi({x_{n_k}+\mathscr N_\omega}),(\underbrace{0,\cdots,0,x+\mathscr N_{\omega_j}}_{j-th~ term},0,\cdots
		 	)\right\rangle\nonumber\\
		 	&=\left\langle \Psi(\tilde{u}),(
		 		\underbrace{0,\cdots,0,x+\mathscr N_{\omega_j}}_{j-th~ term},0,\cdots
		 	)\right\rangle\nonumber\\
		 	&=(\phi_{j}(\tilde{u}),x+\mathscr N_{\omega_j})_{\omega_j}.
		 	\end{align}
		 	Equation \eqref{eq weak} shows that $x_{n_k}+\mathscr N_{\omega_j}$ is weakly convergent to $\phi_j(\tilde{u})$ for all $j\geq 1$. It follows from \eqref{eq part part} that $\phi_j(\tilde{u})=\phi_j(\tilde{z})$ for all $j\geq 1$. Hence, $\Psi(\tilde{u})=\Psi(\tilde{z})$. Since $\Psi$ is an isometric, we conclude that $\tilde{z}=\tilde{u}\in \overline{\iota_\omega(\mathscr L)}$.
	\end{proof}
	
\begin{corollary}
Let $\mathscr H$ and $\mathscr K$ be closed submodules of $\mathscr E$ such that $\mathscr H\cap \mathscr K$ is an orthogonally complemented submodule. Let $\omega_j$'s be positive linear functionals on $\mathscr A$ for each $j\geq 1$. If $(\mathscr H\cap \mathscr K)_{\omega_j}=\mathscr H_{\omega_j}\cap \mathscr K_{\omega_j}$ for each $j$, then 
\begin{equation}
(\mathscr H\cap \mathscr K)_\omega=\mathscr H_\omega\cap \mathscr K_\omega\,, \quad(\omega\in \sigma\mbox{-co}\{\omega_1,\omega_2,\ldots\})\,.
\end{equation}
\end{corollary}
\begin{proof}
It is enough to show that $\left((\mathscr H\cap\mathscr K)_\omega\right)^\perp\subseteq\left(\mathscr H_\omega\cap\mathscr K_\omega\right)^\perp$. Let $\tilde{z}\in \left((\mathscr H\cap\mathscr K)_\omega\right)^\perp$. By Lemma \ref{lemma Mesland}, $\tilde{z}\in \left((\mathscr H\cap\mathscr K)^\perp\right)_\omega$. Theorem \ref{thmain} yields the existence of a sequence $\{x_n\}\subseteq(\mathscr H\cap \mathscr K)^\perp$ such that $\{x_n+\mathscr N_\omega\}$ is bounded and 
\[
\lim_{n\to\infty}(x_n+\mathscr N_{\omega_j})=\phi_j(\tilde{z})\quad(j\geq 1).
\]
Hence, $\phi(\tilde{z})\in \left((\mathscr H\cap\mathscr K)^\perp\right)_{\omega_j}=\left((\mathscr H\cap\mathscr K)_{\omega_j}\right)^\perp$.
Let $\tilde{u}\in (\mathscr H_\omega\cap\mathscr K_\omega)$ be arbitrary. We see, by Theorem \ref{thmain}, that $\phi_j(\tilde{u})\in \mathscr H_{\omega_j}\cap\mathscr K_{\omega_j}=(\mathscr H\cap\mathscr K)_{\omega_j}$. So, $(\phi_j(\tilde{z}),\phi_j(\tilde{u}))_{\omega_j}=0$ for all $j\geq 1$. Let $\Psi$ be as in the proof of Theorem \ref{thmain}. Then, 
\[
(\tilde{z},\tilde{u})_\omega=\langle \Psi(\tilde{u}),\Psi\tilde{z})\rangle=\sum_{j=1}^{\infty}\lambda_j(\phi_j(\tilde{z}),\phi_j(\tilde{z}))_{\omega_j}=0
\]
This ensures that $\tilde{z}\in (\mathscr H_\omega\cap \mathscr K_\omega)^\perp$.
\end{proof}
\section{Results on Hilbert $\mathscr A$-modules over $W^*$-algebra}
Let $X\otimes_\pi Y$ be the tensor product $X\otimes Y$ endowed with the projective norm 
\[
\pi(u)=\inf\left\{\sum_{k=1}^n\|x_k\|\|y_k\|:u=\sum_{k=1}^n x_k\otimes y_k\right\}.
\]
 We denote its completion by $X\hat{\otimes}_\pi Y$.
 The Banach
space $X\hat{\otimes}_\pi Y$ will be referred to as the projective tensor product of the Banach
spaces $X$ and $Y$.
\begin{theorem}\cite[Proposition 2.8]{Ray}\label{theorem tensor}
	Let $X$ and $Y$ be Banach spaces. Let $u\in X\hat{\otimes}_\pi Y$ and $\epsilon>0$.
	Then, there exist bounded sequences $\{x_n\}$ and $\{y_n\}$ in respectively $X$ and $Y$ such that
	the series $\sum_{n=1}^\infty x_n\otimes y_n$ converges to $u$ and
	\[
	\sum_{n=1}^{\infty}\|x\|\|y_n\|<\pi(u)+\epsilon.
	\]
\end{theorem}
According to Theorem \ref{theorem tensor}, for each $u\in X\hat{\otimes}_\pi Y $ we have 
\[
\pi(u)=\inf\left\{\sum_{n=1}^\infty\|x_n\|\|y_n\|:\sum_{n=1}^\infty\|x_n\|\|y_n\|<\infty,u=\sum_{n=1}^\infty x_n\otimes y_n\right\}.
\]
Let $\mathscr A$ be a $W^*$-algebra with the predual $\mathscr A_*$. Let $\mathscr E$ be a Hilbert $\mathscr A$-module. We say that $\mathscr E$ is self-dual as defined in \cite{FRA, MT2}, if $\mathscr E'=\hat{\mathscr E}$, where 
\[
\mathscr E'=\{\tau:\mathscr E\longrightarrow \mathscr A:\tau \mbox{~is bounded~and~} \mathscr A\mbox{-linear}\}
\] 
and $\hat{\mathscr E}=\{\hat{x}:x\in \mathscr E \mbox{~and}, \hat{x}(y)=\langle x, y\rangle \mbox{~for ~all ~}y\in \mathscr E\}$.
The following result is well-known. 
\begin{proposition}\label{conjuate }\cite[Propsition 3.8]{Pa}
	Let $\mathscr A$ be a $W^*$-algebra. Let $\mathscr E$ be a self-dual Hilbert $\mathscr A$-module. Then, $\mathscr E$ is a conjuate space.
\end{proposition}
 Making use of Proposition \ref{conjuate }, we consider the weak$^*$-topology on $\mathscr E$. The set of all weak$^*$-continuous linear functionals on $\mathscr E$ is a subspace of $\mathscr{A}_*\hat{\otimes}_\pi \mathscr Y$ where $\mathscr Y$ is the linear space $\mathscr X$ with twisted scalar multiplication (i.e., $\lambda . x =\bar{\lambda}x$
for $\lambda\in \mathbb{C},x\in \mathscr Y$).
 According to \cite[Proposition 2.8]{Ray}, we can represent each $u\in \mathscr{A}_*\hat{\otimes}_\pi \mathscr Y$ as $\sum_{j=1}^\infty \lambda_j\tau_j\otimes y_j$ where $\lambda_j>0$, $\sum_{j=1}^\infty\lambda_j<\infty$, and $\|\tau_j\|=\|y_j\|=1$. To achieve our main result, we need to establish some lemmas.

\begin{lemma}\label{lem inequality}
	Let $\mathscr E$ be a Hilbert $\mathscr A$-module over a $C^*$-algebra $\mathscr A$. Let $\sigma(\cdot,\cdot)$ be an $\mathscr A$-valued semi-inner product on $\mathscr E$. Let $x_0,z_i\in \mathscr E$ and $\lambda_i>0$ , $(i=1,\ldots ,n)$ be scalars such that $\sum_{i=1}^n\lambda_i=1$. Then, 
	\[
	\sum_{i=1}^n\lambda_i\sigma(z_i-x_0,z_i-x_0)\geq \sigma\left(\sum_{i=1}^n\lambda_iz_i-x_0,\sum_{i=1}^n\lambda_iz_i-x_0\right)\,.
	\] 
\end{lemma}
\begin{proof}
	Since $\sigma(y-x,y-x)\geq 0$, we have $\sigma(y,x)+\sigma(x,y)\leq \sigma(y,y)+\sigma(x,x)$. Hence,
	\begin{align*}
	\sum_{r,s=1}^n\lambda_r\lambda_s\sigma(z_r,z_s)&=\sum_{r=1}^n\lambda_r^2\sigma(z_r,z_r)+\sum_{1\leq r<s\leq n}\lambda_r\lambda_s\left(\sigma(z_r,z_s)+\sigma(z_s,z_r)\right)\\
	&\leq \sum_{r=1}^n\lambda_r^2\sigma(z_r,z_r)+\sum_{1\leq r<s\leq n}\lambda_r\lambda_s\left(\sigma(z_r,z_r)+\sigma(z_s,z_s)\right)\\
	&=\sum_{r=1}^n\lambda_r\sigma(z_r,z_r).
	\end{align*}
	This entails that 
	\begin{align*}
	\sum_{i=1}^n\lambda_i\sigma(z_i-x_0,z_i-x_0)&=\sigma(x_0,x_0)-\sum_{i=1}^n\lambda_i\left(\sigma(z_i,x_0)+\sigma(x_0,z_i)\right)+\sum_{i=1}^n\lambda_i\sigma(z_i,z_i)\\
	&\geq \sigma(x_0,x_0)-\sum_{i=1}^n\lambda_i\left(\sigma(z_i,x_0)+\sigma(x_0,z_i)\right)+\sum_{r,s=1 }^n\lambda_r\lambda_s\sigma(z_r,z_s)\\
	&=\sigma\left(\sum_{i=1}^n\lambda_iz_i-x_0,\sum_{i=1}^n\lambda_iz_i-x_0\right).
	\end{align*}
\end{proof}
\begin{lemma}\label{lem characterzation state}
Let $\mathscr A$ be a $C^*$-algebra	and let 
$$\mathscr B=\oplus_{j=1}^\infty\mathscr{A}=\{(a_j)_j:a_j\in \mathscr A_j \mbox{~and ~}\sup_j\|a_j\|<\infty\}$$ be equipped with the norm $\|(a_j)_j\|=\sup_j\|a_j\|$. Let $\lambda_j>0$ and $\sum_{j=1}^\infty \lambda_j=1$. If $\{\omega_j\}$ is a bounded sequence in $\mathscr A$, then $\omega\left((a_j)_j\right)=\sum_{j=1}^\infty\lambda_j\omega_j(a_j)$ defines a bounded functional on $\mathscr B$. 
\end{lemma}
\begin{proof}
	It is evident that $\omega$ is a well-defined linear functional. We demonstrate that $\omega$ is bounded. Let $\lim_n(a_{n,j})_j=(a_j)_j$ and let $\epsilon>0$ be arbitrary. There exists $n_0>0$ such that for each $j$ and $n>n_0$ we have 
	\[
	\|a_{n,j}-a_j\|<\epsilon\,.
	\]
	Since $\{\omega_j\}$ is a bounded sequence, there exists $m>0$ such that $\|\omega_j\|\leq m$ for each $j\geq 1$. Hence, for each $n>n_0$, we have 
	\begin{align*}
	|\omega\left((a_{n,j})_j-(a_j)\right)|&=\left|\sum_j\lambda_j\omega_j(a_{n,j}-a_j)\right|\\
	&\leq \sum_j|\lambda_j\omega_j(a_{n,j}-a_j)|\\
	&\leq m\sum_j\lambda_j\|(a_{n,j}-a_j)\|<m\epsilon\,.
	\end{align*}
	Therefore, $\lim_{n\to\infty}\omega\left((a_{n,j})_j\right)=\omega\left((a_j)_j\right)$. This ensures that $\omega$
 is bounded.\end{proof}

Let $\mathscr A$ be a $C^*$-subalgebra of $\mathcal{B}(\mathscr H)$, where $\mathscr H$ is a Hilbert space. We say that a state $\omega$ on $\mathscr A$ is a vector state if there exists a vector $x\in \mathscr H$ such that $\omega(v)=\langle vx,x\rangle$ for all $v\in \mathscr A$. The next result may be well-known. However, we present a proof for it.

\begin{theorem}\label{sigma convex of von Neuamann alegebra}
	Let $\mathscr A$ be a $W^*$-algebra. Then, the normal state space of $\mathscr A$ is the $\sigma$-convex hull of its vector states. 
\end{theorem}
\begin{proof}
	Let $\tau$ be a normal sate on $\mathscr A$. By \cite[Theorem 4.2.10]{R}, there exists a trace class operator $u$ such that $\tau(v)=\text{tr}(uv)$. Consider an orthonormal sequence $\{\eta_j\}$ and real numbers $\mu_j$ such that $\sum_{j=1}^\infty \mu_j=1$ and $u(x)=\sum_{j=1}^\infty\mu_j\langle \eta_j,x\rangle\eta_j$. Then, 
	\[
	\tau(v)=\text{tr}(uv)=\sum_{j=1}^\infty \langle uv(\eta_j),\eta_j\rangle= \sum_{j=1}^\infty \mu_j\langle v(\eta_j),\eta_j\rangle\,.
	\] 
\end{proof}
\begin{theorem}\label{separation}
Let $\mathscr L$ be a closed submodule of a self-dual Hilbert $\mathscr A$-module $\mathscr E$ over a $W^*$-algebra $\mathscr A$ such that $\mathscr E\backslash \mathscr L$ has nonempty interior with respect to the weak$^*$-topology. Then, there exists a normal state $\omega$ such that $\iota_\omega (\mathscr L)$ is not dense in $\mathscr E_\omega $. In particular, $\iota_\omega (\mathscr L)^\perp\neq \{0\}$.
\end{theorem}
\begin{proof}
	Let $x_0\in \mathscr E\backslash \mathscr L$ be a vector that is not in the weak$^*$-closure of $\mathscr L$. The Hahn Banach theorem implies that for each $j$, there exist a linear functional $\omega_j\in \mathscr{A}_*$, an element $y_j\in \mathscr E$, and scalars $\lambda_j>0$ such that $\|\omega_j\|=\|y_j\|=1$, $\sum_{j=1}^\infty\lambda_j<\infty$, as well as
\begin{align*}
\sum_{j=1}^\infty\lambda_j\omega_j\langle x_0,y_j\rangle=1\quad {\rm and} \quad \sum_{j=1}^\infty \lambda_j\omega_j\langle l,y_j\rangle=0\quad(l\in \mathscr L_1)\,.
\end{align*}
In particular, it follows that
\begin{equation}\label{eq separation L}
\sum_{j=1}^\infty\lambda_j\omega_j\langle x_0-l,y_j\rangle=1\,,
\end{equation}
for all $l\in \mathscr L$.
 Let $\mathscr{B}=\oplus_j\mathscr A$. Then, $\mathscr B$ is a von-Neumann algebra. Let $\tau$ be the topology on $\oplus_j\mathscr A$ induced by the following functionals 
 \[
 \omega: \mathscr B\to \mathbb{C},\quad\omega\left((a_j)_j\right)=\sum_{j=1}^\infty \lambda_j\omega_j(a_j)\qquad(a_j\in \mathscr A),
 \]
 where $\omega_j\in \mathscr{A}_*$ and $\{\|\omega_j\|\}$ is bounded. Set $$\mathscr I:=\{(\langle y_j,l-x_0\rangle\langle l-x_0,y_j\rangle)_j : l\in \mathscr L\}\subseteq \mathscr B.$$
 Equality \eqref{eq separation L} ensures that $0$ does not belong to the closure $\mathscr L$ in the topology of $\tau$. In fact, if $\lim_\alpha(\langle y_j,l_\alpha-x_0\rangle\langle l_\alpha-x_0,y_j\rangle)_j=0$ for some net $\{l_\alpha\}\subseteq \mathscr L$, then $\lim _\alpha\|(\langle l_\alpha-x_0,y_j\rangle)_j\|=0$. It follows from Lemma \ref{lem characterzation state} that $\lim_\alpha\sum_{j=1}^\infty\lambda_j\omega_j(\langle l_n-x_0,y_j\rangle=0$, which contradicts \eqref{eq separation L}. 
 
 Next, $0\not \in \overline{\mbox{co}(\mathscr I)}^\mathfrak{T}$, since for each $\{l_1,\cdots,l_n\}\subseteq \mathscr L$, $\alpha_i>0$ and $\sum_{i=1}^n\alpha_i=1$, by employing Lemma \ref{lem inequality} for $\mathscr A$-valued semi-inner product $\sigma$ on $\mathscr \oplus_j E$ defined by 
 $
 \sigma\left((u_j)_j,(v_j)_j\right)=\left(\langle y_j, u_j\rangle\langle v_j,y_j \rangle\right)_j,
 $
 we have 
 \[
 \sum_{i=1}^n\alpha_i\left(\langle y_j,l_n-x_0\rangle\langle l_n-x_0,y_j\rangle\right)_j\geq \left\langle y_j,\sum_{i=1}^n\alpha_il_n-x_0\right\rangle\left\langle\sum_{i=1}^n\alpha_i l_n-x_0,y_j\right\rangle.
 \]
 Thus, $0\not \in \overline{\mbox{co}(\mathscr I)}^\mathfrak{T}$. From the Hahn-Banach Theorem, we can deduce that there exist $\epsilon>0$ and $\tau=(\tau_j)_j$, where $\tau_j\in\mathscr{A}_*$ such that ${\rm Re\,}\tau((a_j)_j)>\epsilon$ for each $(a_j)_j\in \overline{\mbox{co}(\mathscr I)}^\mathfrak{T}$. On the other hand, there are positive normal functionals $\tau_{1,j},\tau_{2,j},\tau_{3,j}$, and $\tau_{4,j}$ such that $\tau_j=\tau_{1,j}-\tau_{2,j}+i(\tau_{3,j}-\tau_{4,j})$. Thus, 
 $\sum_j\lambda_j\tau_{1,j}(\langle y_j,l-x_0\rangle\langle l-x_0,y_j\rangle)>\epsilon$ for each $l\in \mathscr L$. Since $\langle y_j,l-x_0\rangle\langle l-x_0,y_j\rangle\leq \|y_j\|^2\langle l-x_0,l-x_0\rangle$ we have 
 \begin{align*}
 (\sum_{j=1}^\infty\lambda_j\tau_{1,j})(\langle l-x_0,l-x_0\rangle)= \sum_{j=1}^\infty\lambda_j\tau_{1,j}(\langle l-x_0,l-x_0\rangle)>\epsilon.
\end{align*}
Hence, $\tau=\frac{\sum_{j=1}^\infty\lambda_j\tau_{1,j}}{\|\sum_{j=1}^\infty\lambda_j\tau_{1,j}\|}$ is a normal state such that $x_0+\mathscr N_\tau\not \in \overline{\iota_\tau(\mathscr L)}$. 
\end{proof}
The next result is as follows.
\begin{theorem}
Let $\mathscr L$ be a closed submodule of a self-dual Hilbert $\mathscr A$-module $\mathscr E$ over a $W^*$-algebra $\mathscr A$ such that $\mathscr E\backslash \mathscr L$ has a nonempty interior in the weak$^*$-topology. Then, there exist vector states $\omega_1,\ldots,\omega_n$ such that $\iota_\omega (\mathscr L_1)$ is not dense in the unit ball of $\mathscr E_\omega $ for each $\omega\in \mbox{co}\{\omega_1,\ldots,\omega_n\}$.
\end{theorem}
\begin{proof}
It follows from Theorem \ref{separation} that there exists a unit vector $x_0\in \mathscr E\backslash \mathscr L$ and a normal state $\tau $ such that $x_0+\mathscr N_\tau\not\in \overline{\iota_\tau(\mathscr L)}$. Theorem \ref{sigma convex of von Neuamann alegebra} ensures that $\tau=\sum_{j=1}^\infty \lambda_j\omega_j$ for some vector states $\omega_j$ and scalars $\lambda_j$ such that $\sum_{j=1}^\infty \lambda_j=1$. Let $\mathscr L_1$ be the unit ball of $\mathscr L$. Let 
\[
B_{j,n}=\{\tilde{z}\in\overline{\iota_{\tau}(\mathscr L)}: \|x_0+\mathscr N_{\omega_j}-\phi_j(\tilde{z})\| >\frac{1}{n}\}\quad(j,n\geq 1)\,.
\]
Employing Theorem \ref{thmain}, we get $\overline{\iota_\tau(\mathscr L_1)}\subseteq \cup_{j,n=1}^\infty B_{j,n}$. Since $\mathscr L_1$ is a weak$^*$-compact set, we can conclude that $\overline{\iota_\tau(\mathscr L_1)}$ is also a weak$^*$-compact set. Hence, there exists $\delta>0$ and $n_0\geq 0$ such that 
\[
\omega_{j_k}\langle x-x_0,x-x_0\rangle\geq \delta \quad(x\in \mathscr L_1, k=1,\ldots, n_0)\,.
\]
Let $\omega=\sum_{k=1}^{n_0}\gamma_{k}\omega_{j_k}$ and $\sum_{k=1}^{n_0}\gamma_{k}=1$. We conclude from Theorem \ref{thconvex} that 
$x_0+\mathscr N_{\omega}\not\in \overline{\iota_\omega(\mathscr L_1)}$
\end{proof}
In the following, we provide an example that meets the hypotheses of Theorem \ref{separation}.
\begin{example}
	Let $\mathscr A$ be the $W^*$-algebra of all bounded linear operators on a separable Hilbert space $\mathscr H$. Consider $\mathscr E=\mathscr A$ as a Hilbert $\mathscr A$-module. Let $\mathscr F=\mathscr A\oplus \mathscr A$. If $\mathscr{M}$ be a proper orthogonally complemented submodule of $\mathscr E$, then set $\mathscr L=\mathbb{K}(\mathscr H)\oplus \mathscr{M}$, where $\mathbb{K}(\mathscr H)$ is the $C^*$-algebra of all compact operators on $\mathscr H$ \cite{DAN}. By employing \cite[Theorem 4.1.15]{R}, we observe that $\mathbb{K}(\mathscr H)$ is not orthogonally complemented in $\mathscr A$. Thus, $\mathscr L$ is not orthogonally complemented in $\mathscr F$. For each nonzero $x_0\in \mathscr{M}^\perp$, it can be observed that $0\oplus x_0$ is an interior point of $\mathscr F\backslash \mathscr L$ in the weak$^*$-topology. In fact, if a net $\{k_\alpha\oplus x_\alpha\}_\alpha$ converges to $0\oplus x_0$, then for a normal state $\omega$ on $\mathscr A$ such that $\omega\langle x_0,x_0\rangle=\|x_0\|^2$ we have 
	\[
	0=\lim_{\alpha}\omega\langle k_\alpha\oplus x_\alpha,0\oplus x_0\rangle=\omega\langle 0\oplus x_0,0\oplus x_0\rangle=\|x_0\|^2\,,
	\]
	which leads to a contradiction.
\end{example}
 
 \medskip
 \section*{Disclosure statement}
 
 On behalf of the authors, the corresponding author states that there is no conflict of interest. Data sharing is not applicable to this paper as no datasets were generated or analyzed during the current study.
 \medskip
%%=========================================================================

	\end{document}